\documentclass[12pt]{amsart}

\usepackage[latin2]{inputenc}
\usepackage{t1enc}
\usepackage[mathscr]{eucal}
\usepackage{pgfplots}
\pgfplotsset{width=7cm,compat=1.8}
\usepackage{pgfplotstable}
\usepackage{caption}
\usepackage{graphicx}
\usepackage{graphics}
\numberwithin{equation}{section}
\usepackage[margin=3.5cm]{geometry}
\usepackage{hyperref} 
\usepackage{tikz}
\usepackage{pgfplots}
\pgfplotsset{ticks=none, compat=newest}
\usepackage{chngcntr}
\counterwithout{figure}{section}
\pgfplotsset{soldot/.style={color=black,only marks,mark=*}}
\pgfplotsset{holdot/.style={color=black,fill=white,only marks,mark=*}}
\usepackage{mathtools,array}
\usepackage{tikz}
\usepackage{calrsfs}
\usepackage{stmaryrd}
\usepackage{breqn}
\usepackage{verbatim}
\usetikzlibrary{calc,intersections,}
\hypersetup{
 colorlinks=true,       
    linkcolor=cyan,          
    citecolor=green,        
    filecolor=magenta,      
   urlcolor=cyan           
   }

\allowdisplaybreaks
\newcommand{\eq}{\begin{equation}}
\newcommand{\en}{\end{equation}}

\DeclareSymbolFont{bbold}{U}{bbold}{m}{n}
\DeclareSymbolFontAlphabet{\mathbbold}{bbold}
\newcommand{\ind}{\mathbbold{1}}

\makeatletter

\newcommand{\Rmnum}[1]{\expandafter\@slowromancap\romannumeral #1@}
\makeatother

\newtheorem{thm}{Theorem}[section]
\newtheorem{prop}[thm]{Proposition}
\newtheorem{lem}[thm]{Lemma}

\theoremstyle{definition}
\newtheorem{remark}[thm]{Remark}
\newtheorem{defn}[thm]{Definition}

\newtheorem{conjecture}[thm]{Conjecture}
\newtheorem{notation}[thm]{Notation}
\numberwithin{equation}{section}

\newcommand{\ed}{\,{\buildrel d \over =}\,}

\renewcommand{\and}{ \quad \text{and} \quad }

\newcommand{\bl}{\begin{lemma}}
\newcommand{\el}{\end{lemma}}
\newcommand{\bp}{\begin{proposition}}
\newcommand{\ep}{\end{proposition}}
\newcommand{\bcor}{\begin{corollary}}
\newcommand{\ecor}{\end{corollary}}
\newcommand{\bth}{\begin{theorem}}
\newcommand{\et}{\end{theorem}}
\newcommand{\be}{\begin{equation}}
\newcommand{\ee}{\end{equation}}
\newcommand{\bal}{\begin{align}}
\newcommand{\eal}{\end{align}}
\newcommand{\bi}{\begin{itemize}}
\newcommand{\ei}{\end{itemize}}

\newcommand{\G}{\Gamma}

\renewcommand{\i}{\infty}

\newcommand{\bE}{{\mathbb E}}
\newcommand{\bN}{{\mathbb N}}
\newcommand{\bP}{{\mathbb P}}
\newcommand{\bR}{{\mathbb R}}

\newcommand{\bZ}{{\mathbb Z}}

\newcommand{\cG}{{\mathcal G}}

\newcommand{\cU}{{\mathcal U}}
\newcommand{{\cA}}{{\mathcal A}}

\newcommand{\cN}{{\mathcal N}}

\newcommand{\cF}{{\mathcal F}}
\newcommand{\cH}{{\mathcal H}}

\begin{document}
\author{Mehdi Ouaki}
\address{Department of Statistics \#3860 \\ 451 Evans Hall \\ University of California at Berkeley \\ Berkeley, CA 94720-3860 \\ USA}
\email{mouaki@berkeley.edu}
\author{Jim Pitman}
\address{Department of Statistics \#3860 \\ 367 Evans Hall \\ University of California at Berkeley \\ Berkeley, CA 94720-3860 \\ USA}
\email{pitman@stat.berkeley.edu}


\date{\today}

\title{Markovian structure in the concave majorant of Brownian motion}
\maketitle

\begin{abstract}{ The purpose of this paper is to highlight some hidden Markovian structure of the concave majorant of the Brownian motion. Several distributional identities are implied by the joint law of a standard one-dimensional Brownian motion $B$ and its almost surely unique concave majorant $K$ on $[0,\infty)$. 
In particular, the one-dimensional distribution of $2 K_t - B_t$ is that of $R_5(t)$, where $R_5$ is a $5-$dimensional Bessel process with $R_5(0) = 0$. The process $2K-B$ shares a number of other properties with $R_5$, and we conjecture that it may have the distribution of $R_5$. We also describe the distribution of the 
convex minorant of a three-dimensional Bessel process with drift.
}
\end{abstract}

\section{Introduction}
There is a rich literature on minorants of stochastic processes such as Brownian motion and L\'evy processes. 
initiated by Groeneboom in \cite{gro83}, and continued in \cite{pit83,cin92,ber00,cd01,suidan,abrampitm,bravpitm,pitmanross,MR2798000}  
with applications to problems in statistics such as isotonic regression \cite{isoto}. 
An analogous study of Lipschitz minorants of L\'evy processes was taken up in \cite{abramevans,evansouaki}. 
The purpose of this paper is to develop a deeper understanding of Markovian properties of the concave majorant of Brownian motion. \\

For any real-valued function $f$ defined on a domain $U$, we say that $c$ is its concave majorant if $c$ is the minimal concave function such that $c(u) \ge f(u)$ for all $u \in U$. It was shown by Groeneboom in \cite{gro83} that
a standard one-dimensional Brownian motion $B$ admits almost surely a unique concave majorant $K$ on the domain $[0,\infty)$, with the following properties. The process $(K(t))_{t \ge 0}$ is an increasing piecewise-linear function with infinitely many linear segments which accumulate only at zero and at infinity, and the process of slopes and lengths of these segments is a Poisson point process.
Moreover, conditionally on the concave majorant $K$, the process $K-B$ is a concatenation of independent Brownian excursions between the vertices of $K$, where $t$ is a vertex of $K$ if $K(t)=B(t)$, or equivalently if 
$K'(t+)<K'(t-)$ 
where $K'$ is the right-hand derivative of $K$. \\

In the tradition of convex analysis, working with either positions or slopes gives two dual perspectives of a convex function. 
In our setting, we can either consider the concave majorant at a fixed time $t>0$, or we can fix a slope $\mu >0$ and consider the concave majorant at the random time $\sigma_{\mu}:=\text{argmax}^{+}_{t \ge 0} \{ B(t)-\mu t \}$. 
These two points of view offer complementary information about the concave majorant, as discussed further in the literature cited above.

Before stating our main results, let us introduce some notation 
\begin{notation}
For a random variable $X$ taking values in $\bR^{d}$, with a probability density with respect to Lebesgue measure on $\bR^d$, we denote that probability density function by $f_{X}$. 
Similarly, for two random variables $X$ and $Y$ both taking values in Euclidean spaces (not necessarily of the same dimension), the notation
$f_{X|Y=y}(\cdot)$ is used to denote a regular conditional probability density of $X$ given $Y=y$.
\\

For $a,b>0$, 
let $\beta_{a,b}$ denote a random variable with the $\text{beta}(a,b)$ density 
\[
f_{\beta_{a,b}}(u)=\frac{\Gamma(a+b)}{\Gamma(a)\Gamma(b)} u^{a-1} (1-u)^{b-1} \ind_{\{0<u<1\}}
\]
where $\Gamma$ is the Gamma function. For $k \in \bN$, we denote by $\chi_k^2$ a random variable with the chi-squared$(k)$ or gamma$(k/2,1/2)$ distribution of  the sum of squares of $k$ independent
standard Gaussian variables, and write $\chi_k$ for the positive square root of $\chi_k^2$, with chi$(k)$ distribution.
Finally, $\phi$  is the probability density  and $\bar{\Phi}$ is the right-tail probability) of a standard normal random variable:
\[
\phi(x) =\frac{1}{\sqrt{2 \pi}} \exp \left(-\frac{x^2}{2} \right) , \qquad  \bar{\Phi}(x)= \int_{x}^{\i} \phi(t)dt ~~~,x \in \bR.
\]
For $\mu \in \bR$, we use the notation $(B_{\mu}(t):=B(t)+\mu t)_{t \ge 0}$ for a Brownian motion with drift parameter $\mu$ and unit variance parameter. 
For each level $y>0$ and drift $\mu>0$, define the first and last passage times
\[
T_{\mu,y} : = \inf\{ t >0 : B_{\mu}(t)=y \} \text{ and } G_{\mu,y}:=\sup \{ t > 0 : B_{\mu}(t)=y \}.
\]
So $T_{\mu,y}$ has the inverse Gaussian distribution with parameters $(\mu,y)$, whereas $G_{\mu,y}$ has the size-biased inverse Gaussian distribution with parameters $(\mu,y)$. 
Their respective densities are denoted by $f_{\mu,y}$ and $f^{*}_{\mu,y}$ and given by
\[
f_{\mu,y}(t):=f_{T_{\mu,y}}(t)= \frac{y}{\sqrt{2 \pi t^3}} \exp \left(-\frac{(y-\mu t)^2}{2t} \right )\ind _{\{t>0\}}
\]
and 
\[
f^*_{\mu,y}(t):=f_{G_{\mu,y}}(t)= \frac{ \mu}{y} \, t \, f_{\mu,y}(t) .
\]
For each $\mu \ge 0$, the three-dimensional Bessel process with drift $\mu$ is the unique diffusion with the same law as the radial part of a three-dimensional Brownian motion with drift of magnitude $\mu$. We denote this diffusion by $\text{BES}^{r}(3,\mu)$ 
if it is started at position $r$ at time $0$. A {\em three-dimensional Bessel bridge} $W$ on an interval $[0,a]$ such that $W(0)=x$ and $W(a)=y$ has the law of a three-dimensional Bessel process with drift $0$ started at $x$ at time $0$ and conditioned to take value $y$ at time $a$ (the conditioning is in the sense of regular conditional probability). 
We refer to \cite{revuzyor}[Chapter 3,6,11] for further information about these processes, which appear in some of our statements and proofs.
\end{notation}
\subsection{Fixed time analysis}
For $t>0$, let $(G_t,D_t)$ be the left-hand and right-hand vertices of the segment of $K$ straddling $t$. For fixed $t>0$, we almost surely have $0<G_t < t < D_t$. Define $I(t)=K(t)-tK'(t)$ to be the intercept at $0$ of the line extending this segment. As the process $(K,B)$ enjoys 
Brownian scaling, we can restrict discussion to the time $t=1$,
as in the following proposition, which gives the joint density of different quantities related to the Brownian motion and its concave majorant at time $t=1$.

\begin{prop}\label{fixed-time}
The density function of $(K'(1),I(1),K(1)-B(1),\frac{1}{G_1},D_1)$ is 
\begin{align*}
f_5(a,b,y,v,w)&=\sqrt{\frac{2}{\pi^3(v-1)^3(w-1)^3}} ab(w v-1) \times \\
&\exp \left(-\frac{1}{2}\left(b^2 w+2ab+a^2 v +y^2 \frac{w v -1}{(v-1)(w-1)}\right) \right) \ind_{\{w,v>1\}} \ind_{\{a,b,y>0\}}
\end{align*}
In particular, the following marginals take simpler forms.
\begin{itemize}
\item{The joint density of $(K'(1),I(1),K(1)-B(1))$ at (a,b,y) is
\[
f_3(a,b,y)=4y(a+b+y)\phi(a+b+y) \ind_{\{a,b,y>0\}}
\]}
\item{The conditional density of $D_1-1$ at $t$ given $(K'(1)=a,I(1)=b,K(1)-B(1)=y)$ is given by
\[
g(t)=\frac{a}{a+b+y} f_{a,y}(t)+\frac{b+y}{a+b+y} f^{*}_{a,y}(t)
\]
where $f_{a,y}$ and $f_{a,y}^{*}$ are respectively the inverse Gaussian and size-biased inverse Gaussian densities with parameters $(a,y)$.}
\end{itemize}
\end{prop}
\begin{remark}
Some of these identities are implicit in the work of Groeneboom \cite{gro83} in his proofs using probability estimates of events about the position of Brownian bridges relative to deterministic line, 
and some projections of these identities appear also in Carolyn and Dykstra \cite{cd01}, but we give here a simpler proof via the Poisson description. 
\end{remark}
A straightforward consequence of the above proposition is that 
\begin{equation}
2K(1)-B(1) \ed \chi_5
\end{equation}
which coincides with the one-dimensional marginal at time $1$ of a five-dimensional Bessel process. 
Moreover, by integration it is easy to see that the joint density of the pair $(K(1),K(1)-B(1))$ is equal to
\[
f_{(K(1),K(1)-B(1))}(k,y)=4ky(k+y)\phi(k+y)\ind_{\{k,y>0\}}
\]
and so the pair $(K(1),K(1)-B(1))$ is exchangeable. This observation is reminiscent of the counterpart process $(2M(t)-B(t))_{t \ge 0}$, where $M$ is the running maximum of $B$ defined as 
\[
M(t)=\max_{0 \le s \le t} B(s)
\]
Indeed, by results of \cite{pitroger81,pit75} we have that
\[
2M(1)-B(1) \ed \chi_3
\]
and that the pair $(M(1)-B(1),M(1))$ is exchangeable. More is known actually, the process $(2M(t)-B(t))_{t \ge 0}$ has the law of the radial part of a three-dimensional Brownian motion. 
This observation led us to naturally consider the following conjecture:
\begin{conjecture}\label{conjec}
The process $(2K(t)-B(t))_{t \ge 0}$ has the law of the radial part of a five-dimensional Brownian motion
\end{conjecture}
We discuss this conjecture in Section \ref{2k-b}.  Motivated by this conjecture, we were led to explore Markovian properties encoded in the joint law of $B$ and $K$. 
It is tempting to claim that the three-dimensional process $(K'(t),K(t),B(t))_{t \ge 0}$ is Markovian. However this assertion turns out to be false.  Indeed, from Proposition \ref{fixed-time} it is not hard to see that $D_1$ is \textit{not} conditionaly independent of $G_1$ given $(K'(1),K(1),B(1))$. Together with the fact that $G_1$ is measurable with respect to $\sigma \{ B(u), K(u); u \le 1\}$ and $D_1$ is measurable with respect to $\sigma \{ B(u) , K(u); u \ge 1\}$, this implies that the processes $(K'(u),K(u),B(u))_{u \ge 1}$ and $(K'(u),K(u),B(u))_{u \le 1}$ are not conditionally independent given $(K'(1),K(1),B(1))$. Nonetheless, a Markovian structure is present if we include the next vertex of the concave majorant after time $t$, as indicated in the following theorem:
\begin{thm}\label{4-dim}
The process
\[
(\Psi(t):=(K'(t),K(t),K(t)-B(t),D_t-t) , t \ge 0)
\]
is a time-homogenous Markov process. Its semi-group $P_t((a,k,y,w),\cdot)$ is given as follows
\begin{itemize}
\item{If $t<w$, then $P_t((a,k,y,w),\cdot)$ has the law of 
\[
(a,k+at,Z(t),w-t)
\]
where $Z$ is a three-dimensional Bessel bridge between $(0,y)$ and $(w,0)$.}
\item{If $t \ge w$, then $P_t((a,k,y,w),\cdot)$ has the law of 
\[
(a-C'(t-w),k+at-C(t-w),R(t-w)-C(t-w),D_{t-w}^{R}-(t-w))
\]
where $R$ is a $\text{BES}^{0}(3,a)$, $C$ its convex minorant and $D_{u}^{R}$ is the vertex of $C$ after time $u$.}
\end{itemize}
\end{thm}
Similarly to Proposition \ref{fixed-time}, the following proposition gives the joint law of the Bessel process with drift and its convex minorant at a fixed time, which completes the description of the law of the Markov process $\Psi$. 
\begin{prop}\label{drift-fixed}
Fix $\mu>0$. Let $(R(t))_{t \ge 0}$ be a $\text{BES}^{0}(3,\mu)$ and $(C(t))_{t \ge 0}$ its convex minorant. Let $(G_t,D_t)$ be the left-hand and right-hand vertices of $C$ straddling $t$. Then for $(\alpha,u,x,l,y)$ such that $0<\frac{l}{\mu}<u<t<x+u$, $y>0$ and $0<\alpha<\mu$, we have
\begin{align*}
\bP[C'(t)\in d\alpha,G_t \in du, D_t-G_t \in dx, R(G_t) \in dl, R(t)-C(t) \in dy]=\\
 2\frac{\phi((\mu-\alpha)\sqrt{x})}{\sigma_t(x,u)^3 \sqrt{x}} h^{\mu,\mu-\alpha}\left(u,\mu u-l\right)y^2\phi \left(\frac{y}{\sigma_t(x,u)}\right)d\alpha du dx dl dy
\end{align*}
where $\sigma_t(x,u):=\sqrt{\frac{(t-u)(x+u-t)}{x}}$, and 
\begin{align*}
\bP[C'(t) \in d\alpha, G_t=0, D_t-G_t \in dx, R(G_t)=0,R(t)-C(t) \in dy]=\\
2\frac{\phi((\mu-\alpha)\sqrt{x})}{\sqrt{x}}\left(1-\frac{\alpha}{\mu}\right) \frac{y^2}{\sigma_t(x,0)^3} \phi \left(\frac{y}{\sigma_t(x,0)}\right)d\alpha dx dy
\end{align*}
The function $h^{a,b}$ is given by the following expression
\[
h^{a,b}(s,z)=\frac{2b}{a} \bE \left[ \left(\sqrt{\left(a-\frac{z}{s}\right)\left(\frac{z}{s}-b\right)}-\frac{V}{\sqrt{s}}\right)_{+}^2\right]\frac{\phi(\frac{z}{\sqrt{s}})}{\sqrt{s}} \ind_{\{ bs \le z \le as\}}
\]
for $0<b<a$ and where $V \ed \cN(0,1)$ is a standard normal random variable.
\end{prop}

\subsection{Fixed slope analysis}
Working in terms of slopes, we present here some distributional identities which emerge from study of the path of the Brownian motion and its concave majorant
on $[0,\sigma_{\mu}]$ for a fixed slope $\mu>0$.
Let
\begin{align}\label{vertices}
\cdots < \tau_{\mu,1} < \tau_{\mu,0} :=\sigma_{\mu} < \tau_{\mu,-1}<\tau_{\mu,-2}<\cdots
\end{align}
be an enumeration of the vertices of $K$. By Brownian scaling, we can limit ourselves to the case $\mu=1$. Define
\begin{align}\label{defn-rho}
 \tau_n:=\tau_{1,n} \text{ and } \kappa_n:=B(\tau_n) \text{ and } \rho_n:=\kappa_n-\frac{\Delta \kappa_n}{\Delta \tau_n}\tau_n
\end{align}
where $\Delta x_n=x_{n}-x_{n-1}$ for any sequence $(x_n)_{n \in \bN}$. The following theorem states a non-trivial identity in distribution which comes from an analysis of Markovian behavior of
the faces of $K_{[0,\sigma_1]}$, in the same vein as \cite{pitmanross},Corollary 3]\label{ross-pitman}.

\begin{thm}\label{markov-chain}
The following map on $(0,\infty)^3 \times (0,1)$ 
\[
(t,r,q,u) \mapsto \left(u^2(t+q),u(1-u)(t+q)+ur,\frac{r^2 q}{t(t+q)},\frac{r}{r+(1-u)(t+q)}\right)
\]
is one-to-one and preserves the law of
\[
(\chi_{3}^2 \beta_{1,2}^2,\chi_3^2\beta_{1,2}(1-\beta_{1,2}),\chi_1^2,U)
\]
where $\beta_{a,b}$ is $\text{beta}(a,b)$ and $U \ed \beta_{1,1}$ is uniform $[0,1]$,  with
$\chi_{3}, \beta_{1,2}^2$ and $U$ independent.
As a consequence, both the sequences $\left(\frac{\rho_n}{\sqrt{\tau_n}}\right)_{n \ge 0}$ and $\left(\frac{\kappa_n}{\sqrt{\tau_n}}\right)_{n \ge 0}$ are stationary. Moreover, the former sequence is a Markov chain while the latter sequence is not Markov.
\end{thm}
We finish this section by showing that a random scaling of the path of $B$ on $[0,\sigma_{\mu}]$ produces processes that are absolutely continuous with respect to the three-dimensional Bessel process. Let us introduce first some notation
\begin{notation}
Let us denote by $M$ the Mill's ratio function defined as 
$$M(x):=\frac{ \overline{\Phi}(x)} {\phi(x)}.$$ 
For a non-negative continuous function $f$ defined on $[0,1]$, define its minslope as 
\[
\mathfrak{M}(f)=\inf_{0<u \le 1} \frac{f(u)}{u}
\]
Moreover, let $\mathfrak{B}(f)$ be the last time $f$ reaches its minslope, i.e \[ \mathfrak{B}(f):=\sup \{ t \in (0,1] : \frac{f(t)}{t}=\mathfrak{M}(f) \}\]
\end{notation}
We have the following theorem
\begin{thm}\label{meander}
Let $\mu >0$, and define $B_{\mu}(t):=B(t)-\mu t$. Define the two pseudo-meanders
\[
(\tilde{B}^{me}(u) , 0 \le u \le 1):= \left(\frac{B_{-\mu}(\sigma_{\mu})-B_{-\mu}((1-u)\sigma_{\mu})}{\sqrt{\sigma_{\mu}}} ,0 \le u \le 1 \right)
\]
and
\[
(\hat{B}^{me}(u) , 0 \le u \le 1 ):= \left(\frac{B(\sigma_{\mu})-B((1-u)\sigma_{\mu})}{\sqrt{\sigma_{\mu}}} ,0 \le u \le 1 \right)
\]
Then both $\tilde{B}^{me}$ and $\hat{B}^{me}$ are both absolutely continuous with respect to $\text{BES}^{0}(3)$ with Radon-Nidokym derivative given respectively by $\displaystyle 2\frac{M(\tilde{B}^{me}(1))}{\tilde{B}^{me}(1)}$ and $\displaystyle 2\frac{\mathfrak{M}(\hat{B}^{me})}{\hat{B}^{me}(1)}$.
\end{thm}

\section{Fixed time analysis}
We recall the main results of Groeneboom in his study of the concave majorant of Brownian motion. We define first the following process
\[
\tau_{a}:= \sigma_{\frac{1}{a}} = \text{argmax}^{+} \left \{ t \ge 0 : B(t) -\frac{t}{a}  \right\} , a > 0 .
\]
The process $\tau$ encodes all the information about the concave majorant $K$. The following result of Groeneboom gives a full description of the law of $\tau$. We have
\begin{thm}[\cite{gro83}, Theorem 2.1]\label{poisson}
The process $\tau$ is pure-jump with independent nonstationary increments. In particular, the Poisson point process of jumps
\[
\cH := \{(r,\Delta\tau_r) : \Delta \tau_r>0 \}
\]
has intensity measure absolutely continuous with respect to Lebesgue measure with density
\[
\rho(r,t)=\frac{1}{r^2 \sqrt{t}} \,\, \phi \left (\frac{\sqrt{t}}{r} \right)
\]
\end{thm}
The concave majorant $K$ of $B$ is constructed from $\tau$ as concatenation of increasing linear segements with slopes $1/r$ and durations $\Delta \tau_r$ for $(r,\Delta \tau_r) \in \cH$. The joint law of $K$ and $B$ is fully described in the following theorem which is also due to Groeneboom.

\begin{thm}[\cite{gro83}, Theorem 2.2]\label{excursion-decomp}
The standard linear Brownian motion $B$ can be decomposed into the process $\tau$ and independent Brownian excursions. More precisely, conditionally on the process $\tau$, the vertical distance of the Brownian motion to the concave majorant $(K(t)-B(t))_{t \ge 0}$ is a succession of independent Brownian excursions, i.e for any measurable enumeration $(T_i)_{i \in \bZ}$ of jump times of $\tau$, depending only on the process $\tau$, the process $(K(t)-B(t))_{T_i \le t \le T_{i+1}}$ has the same distribution as $(\sqrt{T_{i+1}-T_{i}} \textbf{e} \left(\frac{t-T_{i}}{T_{i+1}-T_i} \right))_{T_i \le t \le T_{i+1}}$ where $\textbf{e}$ is a standard Brownian excursion on $[0,1]$, and these excursions are independent as $i$ varies.
\end{thm}
Combining these two results we can give a quick proof of Proposition \ref{fixed-time}

\begin{proof}[Proof of Proposition \ref{fixed-time}]
Let $S(1):=\frac{1}{K'(1)}$ denote the passage time across the level $1$ for the process $\tau$. Fix $0<u<1<u+x$ and $s,z >0$, then we have
\begin{align*}
\bP[S(1) \in ds, G_1 \in du, D_1-G_1 \in dx, B(G_1) \in dz]&=\\
\bP[B(\tau_{s-}) \in dz, \tau_{s-} \in du, \cH \cap (s,s+ds) \times (x,x+dx) \ne \emptyset]+o(dsdudxdz)&=\\
\bP[B(\tau_{s}) \in dz, \tau_s \in du]\rho(s,x)ds dx+o(dsdudxdz)
\end{align*}
as $\bP[\tau_{s-}=\tau_s]=1$ and $(B(\tau_{s-}),\tau_{s-})$ is measurable with respect to $\cH \cap (0,s)\times \bR$. The joint density of $(B(\tau_{s}),\tau_{s})=(B(\sigma_{\frac{1}{s}}),\sigma_{\frac{1}{s}})$ is known from Williams path decomposition (see \cite{wil}). Hence, using the fact that $K(1)-B(G_1)=K'(1)(1-G_1)$, we have that $I(1):=K(1)-K'(1)=B(G_1)-K'(1)G_1$. Hence, by a standard change of variables we get the joint density of $(K'(1),I(1),\frac{1}{G_1},D_1)$. Finally, using Theorem \ref{excursion-decomp} we get that
\[
(K(1)-B(1) | G_1=u, D_1=v) \ed \sqrt{v-u} ~ \textbf{e} \left(\frac{1-u}{v-u} \right) \ed \sqrt{\frac{(1-u)(v-1)}{v-u}} \chi_3
\]
Putting all this together gives the five-dimensional density. The sub-marginals are easily obtained by integrating out the remaining variables.
\end{proof}
\begin{remark} 
By invariance of the law of $(B,K)$ under time-inversion, we know that 
$$(I(1),K'(1),K(1)-B(1),\frac{1}{G_1},D_1) \ed (K'(1),I_1,K(1)-B(1),D_1,\frac{1}{G_1}),$$
 which can easily be checked from the joint density. 
\end{remark}
Recall that a three-dimensional Bessel process with drift $\mu>0$ is the law of the Brownian motion with drift $\mu$ conditioned to stay positive. By Williams path decomposition, it is also the law of the post-$\sigma_{\mu}$ process $(\mu t -B(\sigma_{\mu}+t)+B(\sigma_{\mu}))_{t \ge 0}$ (see \cite{wil}).  Combining this observation with Theorem \ref{poisson} gives the 
following Poisson description of the convex minorant of a three-dimensional Bessel process with drift.
\begin{prop}\label{bessel-fixed}
Let $(R(t))_{t \ge 0}$ be a three-dimensional Bessel process with drift $\mu>0$. The convex minorant $C$ of $R$ is a concatenation of increasing segments with slopes less than $\mu$. The set of pairs of slopes and time spacings of $C$ is a Poisson point process $\cG$ with intensity measure of density with respect to the Lebesgue measure given by
\[
g(\alpha,t)=\frac{\phi \left(\sqrt{t}(\mu-\alpha) \right)}{\sqrt{t}} \ind_{\{0 < \alpha < \mu\}}
\]
Moreover, conditionally on $C$, the process $R-C$ is again a succession of independent Brownian excursions between the vertices of $C$.
\end{prop}
Before establishing the analogue of Theorem \ref{fixed-time} for the three-dimensional Bessel process with drift, we will give an explicit formula for the density of the increments of the {\em zenith process} $(\sigma_{\mu},B_{\mu}(\sigma_{\mu}))_{\mu>0}$ studied by \c{C}inlar \cite{cin92}. While the Laplace transform of the increment $(\sigma_{b}-\sigma_{a},B(\sigma_{b})-B(\sigma_{a}))$ for $a>b$ is straightforward from Williams path decomposition, inverting it is not obvious, as remarked by \c{C}inlar \cite{cin92}, 
who gave a formula via convolutions of measures. We have the following theorem whose proof is inspired by some of Groeneboom's computations.
\begin{thm}\label{density-zenith}
Fix $0<b<a$, and define the function $h^{a,b}$ on $(0,\i)^2$ by 
\[
h^{a,b}(s,z)=\frac{2b}{a} \bE \left[ \left(\sqrt{\left(a-\frac{z}{s}\right)\left(\frac{z}{s}-b\right)}-\frac{V}{\sqrt{s}}\right)_{+}^2\right]\frac{\phi(\frac{z}{\sqrt{s}})}{\sqrt{s}} \ind_{\{ bs \le z \le as\}}
\]
where $V$ is standard normal.
Then the random variable $(\sigma_{b}-\sigma_{a},B(\sigma_b)-B(\sigma_a))$ has density $h_{a,b}$ on $(0,\infty)^2$ and has a point mass in $(0,0)$ of probability $\frac{b}{a}$.
\end{thm}
Before proving Theorem \ref{density-zenith}, we recall the following lemma which gives an estimate on the probability that a Brownian bridge crosses a deterministic line. We quote it from \cite{hajeksidak}[Page 183]
\begin{lem}\label{bridge-line}
Let $B$ a standard linear Brownian motion, and fix $0<s<t$. Moreover, let $a,b,x,y \in \bR$. Then
\begin{align*}
\bP[ \exists u \in (s,t) : B(u)>au+b | B(s)=x,B(t)=y]=\\
\exp \left(-2\frac{(as+b-x)_{+}(at+b-y)_{+}}{t-s}\right)
\end{align*}
\end{lem}
\begin{proof}[Proof of Theorem \ref{density-zenith}]
Fix $0<t_1<t_2$ and let $h>0$, we will investigate the probability of the event 
\[
A:=\{ \sigma_a \in (t_1,t_1+h), \sigma_b \in (t_2,t_2+h) \}
\]
Let us denote $\mu_1:=a$ and $\mu_2:=b$, the event $A$ can be rewritten as 
\[
A=\{ (\forall s \ge 0) ~~ B(s) - \mu_{i}s < \sup_{t_i \le u \le t_i+h} (B(u)-\mu_{i} u) \text{ for } i=1,2\}
\]
Define $\overline{B}_{t_i,h}=\sup_{t_i \le u \le t_{i}+h} (B(u) -B(t_i))$. We will condition on the value of the random variables $\{B(t_i), B(t_i+h), \overline{B}_{t_i,h}, i=1,2 \}$. In particular, we wish to compute the following regular conditional probability
\[
\bP[A | B(t_i)=x_i, B(t_i+h)-B(t_i)=Y_{i}(h), \overline{B}_{t_i,h}=M_{i}(h)]
\]
where $x_i,Y_i(h),M_{i}(h)$ are fixed real numbers. From the definition of $\sigma_a$ and $\sigma_b$, this probability is equal to zero unless we have that $x_1+b(t_2-t_1)<x_2<x_1+a(t_2-t_1)$. By introducing the point $t_0$ that is the location of the intersection point of the two lines with slope $a$ (resp, $b$) passing through $(t_1,x_1)$ (resp. $(t_2,x_2)$) 
\begin{equation}\label{t_0}
t_0=\frac{(x_2-bt_2)-(x_1-at_1)}{a-b} \in (t_1,t_2)
\end{equation}
Then $A$ can be expressed as
\[
A=\bigcap_{i,j=1}^2 \{ \forall z \in J^{j}_i: B_z -\mu_i z < \sup_{t_i \le u \le t_i +h} (B_u-\mu_i u) \}=:\bigcap_{i,j=1}^2 A_{i,j}
\]
with 
\[
J^1_1=(0,t_1) , ~~J^2_1=(t_1+h,t_0) ,~~ J^1_2=(t_0,t_2) , ~J^2_2=(t_2+h,+\infty)
\]
By further conditioning (in the sense of regular conditional probability) on the value of $B(t_0)$, i.e on the event
\[
D=\{ B(t_0)=x,B(t_i)=x_i, B(t_i+h)-B(t_i)=Y_{i}(h), \overline{B}_{t_i,h}=M_{i}(h) \text{ for } i=1,2 \}
\]
We have that
\[
\bP[A\vert D]=\prod_{i,j=1}^2 \bP[A_{i,j} \vert D]
\]
by the independence of increments of the Brownian motion. Now, it is easy to see that
\[
\overline{B}_{t_i,h}+B(t_i)-\mu_i(t_i+h) \le \sup_{t_i \le u \le t_i +h} (B(u)-\mu_i u) \le \overline{B}_{t_i,h}+B(t_i)-\mu_i t_i
\]
Hence 
\begin{align*}
\overline{r_{i,j}(h)}:=\bP[\exists z \in J^{j}_i : B_z-\mu_i z >\overline{B}_{t_i,h}+B(t_i)-\mu_it_i \vert D] \le   \bP[ A_{i,j}^{\mathsf{c}} \vert D] \le \\
\bP[\exists z \in J^{j}_i : 
 B_z-\mu_i z >\overline{B}_{t_i,h}+B(t_i)-\mu_i(t_i+h)\vert D]=:\underline{r_{i,j}(h)}
\end{align*}
Now using Lemma ~\ref{bridge-line}, we have that
\[
\overline{r_{1,1}(h)}=\exp \left (-2 \frac{(M_1(h)+x_1-at_1)_{+}M_1(h)}{t_1} \right )
\]
and
\[
\underline{r_{1,1}(h)}=\exp \left (-2 \frac{(M_1(h)+x_1-a(t_1+h))_{+}(M_1(h)-a h)_{+}}{t_1} \right )
\]
Hence, by Taylor expansion it follows that
\[
\bP[A_{1,1}^{\mathsf{c}}  \vert D]=1-\frac{2M_1(h)(x_1-at_1)_{+}}{t_1}+F_{1,1}(h,M_1(h)^2)
\]
where $F_{1,1}$ is a deterministic function with $\vert F_{1,1}(x,y) \vert \le K_{1,1}(\vert x \vert +\vert y \vert)$ for a positive constant $K_{1,1}$. Similarly, we find that
\[
\bP[A_{1,2}^{\mathsf{c}}  \vert D]=1-\frac{2(M_1(h)-Y_1(h))(x_1-x_0-a(t_1-t_0))_{+}}{t_0-t_1}+F_{1,2}(h,(M_1(h)-Y_1(h))^2)
\]
and
\[
\bP[A_{2,1}^{\mathsf{c}}  \vert D]=1-\frac{2M_2(h)(x_2-x_0-b(t_2-t_0))_{+}}{t_2-t_0}+F_{2,1}(h,M_2(h)^2)
\]
and
\[
\bP[A_{2,2}^{\mathsf{c}}  \vert D]=1-2b(M_2(h)-Y_2(h))+F_{2,2}(h,(M_2(h)-Y_2(h))^2)
\]
with all the $F_{i,j}$'s deterministic verify $|F_{i,j}(x,y)| \le K_{i,j} (|x|+|y|)$ for some positive constant $K_{i,j}>0$.
Hence by multiplying all these terms together we get
\begin{align*}
\bP[A \vert D]=16\frac{M_1(h)(M_1(h)-Y_1(h))M_2(h)(M_2(h)-Y_2(h))}{t_1(t_0-t_1)(t_2-t_0)} \\ \times (x_1-at_1)_{+}(x_1-x_0-a(t_1-t_0))_{+}(x_2-x_0-b(t_2-t_0))_{+}+\\G(h,M_1(h)^2,M_2(h)^2,(M_1(h)-Y_1(h))^2),(M_2(h)-Y_2(h))^2)
\end{align*}
where $G$ is a deterministic function of $5$ variables such that
\begin{equation}\label{G}
\vert G(x_1,x_2,x_3,x_4,x_5) \vert \le K \sum_{i,j=1}^5 \vert x_i \vert  \sqrt{ \prod_{k \neq i,j} \vert x_k \vert}  
\end{equation}
for $K$ a positive constant. 
Now, let us integrate first with respect to the value of $B(t_0)$. The distribution of $B(t_0)$ given that $B(t_1)=x_1$ and $B(t_2)=x_2$ is that of
\[
\frac{(t_2-t_0)x_1}{t_2-t_1}+\frac{(t_0-t_1)x_2}{t_2-t_1}+\sigma V
\]
where:
\[
\sigma^2=\frac{(t_0-t_1)(t_2-t_0)}{t_2-t_1}>0
\]
and $V \ed \mathcal{N}(0,1)$. After integration and using the expression of $t_0$ from \eqref{t_0}, we get that
\begin{align*}
\bP[A \vert C]=\frac{16bM_1(h)(M_1(h)-Y_1(h))M_2(h)(M_2(h)-Y_2(h))}{t_1(t_2-t_1)}(x_1-at_1)_{+}\\
\times  \bE[\left ( (a-b)\sigma-U \right )_{+}^2] + G(h,M_1(h)^2,M_2(h)^2,(M_1(h)-Y_1(h))^2,(M_2(h)-Y_1(h))^2)
\end{align*}
Finally, we can integrate with respect to $(M_1(h),Y_2(h),M_2(h),Y_2(h))$. Since it can be easily seen that
\[
\bE\left [\overline{B}_{t_i,h}\left (\sup_{t_i \le u \le t_i+h}(B(u)-B(t_i+h))\right ) \Bigg \vert B(t_1)=x_1;B(t_2)=x_2 \right ]=\frac{h}{2}+o(h)
\]
using the joint law of  $(B(h),\max_{0 \le u \le h} B(u))$, and the fact that
\[
\bE[G(h,M_1(h)^2,M_2(h)^2,(M_1(h)-Y_1(h))^2)]=O(h^{\frac{5}{2}})
\]
from \eqref{G}, we get that
\[
\bP[A \vert B_{t_1}=x_1,B_{t_2}=x_2]=\frac{4b(x_1-at_1)_{+}}{t_1(t_2-t_1)}\bE[\left ( (a-b)\sigma-V \right )_{+}^2]h^2+o(h^2)
\]
By multiplying by the joint density of $(B (t_1),B(t_2))$ we get that the joint density of $(\sigma_a,\sigma_b,B(\sigma_a),B(\sigma_b))$ at $(t_1,t_2,x_2,x_1)$ is equal to 
\begin{align*}
f_{(\sigma_a,\sigma_b,B(\sigma_a),B(\sigma_b))}(t_1,t_2,x_1,x_2)=4b \left(\frac{x_1}{t_1}-a \right)_{+} \times \\
\bE\left[\left(\sqrt{(y-b)(a-y)}-\frac{V}{\sqrt{t_2-t_1}}\right)_+^2 \right]
\times \ind_{\{b \le y \le a \}} \frac{\phi \left(\frac{x_1}{\sqrt{t_1}}\right)\phi \left(\frac{x_2-x_1}{\sqrt{t_2-t_1}}\right)}{\sqrt{t_1(t_2-t_1)}}
\end{align*}
where $y:=\frac{x_2-x_1}{t_2-t_1}$. By an obvious change of variable, we see that $(\sigma_a,B(\sigma_a))$ and $(\sigma_b-\sigma_a,B(\sigma_b)-B(\sigma_a))$ are independent (as predicted by Williams path decomposition), and that the density of the latter pair is exactly $h^{a,b}$ that is stated in the Theorem. Here, we use Corollary 2.1 from \cite{gro83},  which stipulates that the two functions
\[
t \mapsto 2a^{2} \bE \left[\left(\frac{X}{a\sqrt{t}}-1\right)_{+}\right] 
\] 
and
\[
t \mapsto \frac{2b}{a-b} \bE \left[\left(\sqrt{\left(\frac{Y}{\sqrt{t}}-b\right)\left(a-\frac{Y}{\sqrt{t}}\right)}-\frac{Z}{\sqrt{t}}\right)_{+}^2 \ind_{\{ b\le \frac{Y}{\sqrt{t}} \le a\}}\right]
\]
are densities, where $X,Y,Z$ are three independent standard normal random variables. It also follows from here that the density $h_{a,b}$ is defective as
\[
\int_{(0,\infty)^2} h^{a,b}(s,z)dsdz=1-\frac{b}{a}
\]
To conclude the proof, it suffices to notice that
\begin{align*}
\bP[\sigma_b=\sigma_a,B(\sigma_b)=B(\sigma_a)]&=\bP[\sigma_b=\sigma_a]\\
&=\bP\left[\text{no jumps for the process }\tau \text{ in } \left(\frac{1}{a},\frac{1}{b}\right)\right ] =\frac{b}{a}
\end{align*}
as the number of jumps $\tau$ on any interval $(x,y)$ is a Poisson random variable with mean $\log(\frac{y}{x})$ from \cite{gro83}[Theorem 2.1].
\end{proof}

The proof of Proposition \ref{bessel-fixed} is now straightforward and goes as follows.
\begin{proof}[Proof of Proposition \ref{bessel-fixed}]
By Williams path decomposition, the process $R$ has the same distribution as
\[
(\tilde{R}(t):=\mu t-B(\sigma_{\mu}+t)+B(\sigma_{\mu}))_{t \ge 0}
\]
whose convex minorant $\tilde{C}$ is given by $\tilde{C}(t)=\mu t-K(\sigma_{\mu}+t)+K(\sigma_{\mu})$. Hence, when $u,l>0$ then the event
\[
\{ \tilde{C}'(t) \in d\alpha, G_t \in du, D_t-G_t \in dx, \tilde{R}(G_t) \in dl \}
\]
is the same as 
\begin{align*}
\{ \text{The process } \tau \text{ has a jump in } \left(\frac{1}{\mu-\alpha},\frac{1}{\mu-\alpha}+\frac{d\alpha}{(\mu-\alpha)^2} \right) \text{ of size in } dx,\\
 \sigma_{\mu-\alpha}-\sigma_{\mu} \in du, \mu(\sigma_{\mu-\alpha}-\sigma_{\mu})-(B(\sigma_{\mu-\alpha})-B(\sigma_{\mu}))\in dl\}
\end{align*}
As we take the restriction of $\tau$ on $(\frac{1}{\mu},\i)$. This latter probability is equal to
\[
\rho \left(\frac{1}{\mu-\alpha},x \right)\frac{dx d\alpha}{(\mu-\alpha)^2} \bP[\sigma_{\mu-\alpha}-\sigma_{\mu} \in du, \mu(\sigma_{\mu-\alpha}-\sigma_{\mu})-(B(\sigma_{\mu-\alpha})-B(\sigma_{\mu})) \in dl]
\]
which by Theorem \ref{density-zenith} gives the density
\[
\rho \left(\frac{1}{\mu-\alpha},x \right)\frac{dx d\alpha}{(\mu-\alpha)^2}h^{\mu,\mu-\alpha}(u,\mu u -l)dudl
\]
Finally, using the fact that conditionally on $C$, $R-C$ is distributed as a concatenation of Brownian excursions between the vertices of $C$, we get that
\[
(R(t)-C(t)| (G_t,D_t)=(x,x+u)) \ed \sigma_t(x,u) \chi_3
\]
gives the desired density. The case when $u=0$ follows similarly.
\end{proof}
We finish this section now by proving Theorem \ref{4-dim}, which 
follows from the description of Groeneboom in Theorem \ref{excursion-decomp} and the independence of increments of the pure-jump Markov process $\tau$. 

\begin{proof}[Proof of Theorem \ref{4-dim}]
Let us fix $t>0$, and enumerate the standard Brownian excursions that constitute the path of $K-B$ such that $\textbf{e}_0$ is the one corresponding to the interval $[G_t,D_t]$, the excursions $\{ \textbf{e}_i : i \ge 1\}$ the ones to the right of $D_t$, and $\{ \textbf{e}_i : i \le -1 \}$ the ones to the left of $G_t$. The process $(\Xi(u) : u \ge t)$ is measurable with respect to the $\sigma$-algebra generated by
\[
\{ \Xi(t), \textbf{e}_0(u)_{u \in \left[\frac{t-G_t}{D_t-G_t},1 \right]}\} \bigvee \{ \textbf{e}_{i}, i \ge 1\} \bigvee \{(\tau_{u+S_t}-\tau_{S_t})_{u \ge 0}\}
\]
where $S_t=\frac{1}{K'(t)}=\inf \{ u > 0 : \tau_u > t\}$ the first passage level across $t$. Whereas, the process $(\Xi(u): u \le t)$ is measurable with respect to
\[
\{ \Xi(t), \textbf{e}_0(u)_{u \in \left[0,\frac{t-G_t}{D_t-G_t} \right]}\} \bigvee \{\textbf{e}_{i}, i \le -1\} \bigvee \{(\tau_{u})_{u \le S_t}\}
\]
Using the fact that $(t,\tau_t)_{t \ge 0}$ is a strong Markov process, and the independence of increments of $\tau$, we have that conditionally on $S_t$, the process $(\tau_{u+S_t}-\tau_{S_t})_{u \ge 0}$ and $(\tau_u)_{u \le S_t}$ are independent. Moreover, by the Markov property of $\textbf{e}_0$ and the form of its transition function, it is not hard to see that $\Xi$ is a time-homogenous Markov process. The description of its semigroup is given by the two cases; In the first case when the time increment is smaller than the remaining time until the next vertex $D_t-t$, this is just a restating of the fact that the distribution of a standard Brownian excursion conditioned on having a specific value at a point in the interior of $[0,1]$ is just a concatenation of two three-dimensional Bessel bridges. The second part follows from Williams path decomposition, as conditionally on $\Xi(t)$, the process $(K'(t)u-B(u+D_t)+B(D_t))_{u \ge 0}$ is a three-dimensional Bessel process with drift $K'(t)$. 
\end{proof}

\section{Fixed slope analysis}
We give in this section the proofs of the two independent results in Theorems \ref{markov-chain} and \ref{meander}. But first, let us recall some previous results from \cite{pitmanross} regarding a sequential description of the faces of Brownian paths. We start with a definition
\begin{defn}
We say that a sequence of random variables $(\tau_n,\rho_n)_{n \ge 0}$ satisfies the $(\tau,\rho)$-recursion if for all $n \ge 0$
\[
\rho_{n+1}=U_n \rho_n \text{ and } \tau_{n+1}=\frac{ \tau_n \rho_{n+1}^2}{\tau_n Z_{n+1}^2+\rho_{n+1}^2}
\]
where $(U_n,Z_n)_{n \ge 0}$ is an i.i.d sequence with common law $ \cU([0,1]) \otimes \cN(0,1)$, and independent of $(\tau_0,\rho_0)$.
\end{defn}
We have the following theorem from \cite{pitmanross}
\begin{thm}[\cite{pitmanross},Corollary 20]\label{ross-pitman}
Let $(B(t))_{t \ge 0}$ be a standard linear Brownian motion. Fix $r>0$, set $\rho_0:=r$ and let $\rho_1>\rho_2>\cdots>0$ be the intercepts at $0$ of the linear segments of the concave majorant of $(B(t), 0 \le t \le T_r)$ where $T_r:=\inf \{ t  \ge 0 : B(t)=r\}$ and let $\tau_0:=T_r>\tau_1>\cdots>0$ denote the decreasing sequence of times $t$ such that $(t,B(t))$ is a vertex of the concave majorant of $(B(t) ,0 \le t \le T_r)$. Then, the sequence of pairs $(\tau_n,\rho_n)_{n \ge 0}$ follows the $(\tau,\rho)$-recursion.
\end{thm}

Let us now prove Theorem \ref{markov-chain}
\begin{proof}[Proof of Theorem \ref{markov-chain}]
Recall that $\tau_0:=\sigma_1$, $\rho_0:=B(\tau_0)-\tau_0$, and the sequence $(\tau_n,\rho_n)_{n \ge 0}$ is defined in equations \eqref{vertices} and \eqref{defn-rho}. Let $(M_n:=\frac{\Delta \kappa_n}{\Delta \tau_n})_{n \ge 1}$ be the decreasing sequence of slopes of $K$ that are greater than $1$. By Groeneboom Poisson description in Theorem \ref{poisson}, it follows that the point process of inverse slopes $\left(\frac{1}{M_n}\right)_{n \ge 1}$ is a Poisson point process on $(0,1)$ with intensity measure of density $\frac{r}{dr}$ with respect to the Lebesgue measure. It follows thus easily by the stick-breaking description that there exists a sequence $(\hat{U}_n)_{n \ge 1}$ of i.i.d uniform random variables on $[0,1]$ such that 
\[
M_n=\prod_{i=1}^{n} \hat{U}_i
\]
Furthermore, conditionally on the sequence $(M_n)_{n \ge 1}$, the sequence of time-spacings $(\Delta \tau_n)_{n \ge 1}$ is distributed as $(\frac{\hat{Q}_n}{M_n^2})_{n \ge 1}$ where $(\hat{Q}_n)_{n \ge 1}$ is an i.i.d sequence of random variables with law $\chi_1^2$, that is independent of the sequence $(\hat{U}_n)_{n \ge 1}$. This follows from the fact that the time-spacing corresponding to an inverse slope $r$ has density $\frac{1}{r\sqrt{t}}\phi \left(\frac{\sqrt{t}}{r} \right)dt$, which is the density of $r^2\chi_1^2$. Thus, as
\[
\tau_1=\sum_{j=2}^{\i} \Delta \tau_j =\sum_{j=2}^{\i} \frac{\hat{Q}_i}{M_i^2}, \text{ and } \kappa_1=\sum_{j=2}^{\i} \frac{\hat{Q}_i}{M_i}
\]
it follows that
\[
(\tau_1,\Delta \tau_1, \kappa_1, \Delta \kappa_1) \ed (U^2 \tau_0, U^2 Q, U \kappa_0, UQ)
\]
where the pair $(U,Q)$ is independent of $(\tau_0,\kappa_0)$ and has law $\cU([0,1]) \otimes \chi_1^2$. In fact, we can say more, if we define the variables 
\begin{equation}
\label{substitution}
\tilde{\tau}_0=\frac{(\Delta \kappa_1)^2}{(\Delta \tau_1)^2}, \tilde{\kappa}_0= \frac{\Delta \kappa_1}{\Delta \tau_1} \kappa_1 , \tilde{U}=\frac{\Delta \tau_1}{\Delta \kappa_1}, \tilde{Q}=\frac{(\Delta \kappa_1)^2}{\Delta \tau_1}
\end{equation}
Then
\[
(\tilde{\tau}_0,\tilde{\kappa}_0,\tilde{U},\tilde{Q}) \ed (\tau_0,\kappa_0,U,Q)
\]
Similarly we define $\tilde{\rho}_0:=\tilde{\kappa}_0-\tilde{\tau_0}$. Now we move on to find the second identity in distribution. From Williams path decomposition, it is clear that 
\[
(B_{-1}(t) , 0 \le t \le \sigma_{1} \vert  \sigma_{1}=\tau_0, B_{-1}(\sigma_{1})=\rho_0 )
\]
has the same distribution as a Brownian first passage bridge from $(0,0)$ to $(\tau_{0},\rho_0)$. Hence, by applying Theorem \ref{ross-pitman} the sequence $(\tau_n,\rho_n)_{n \ge 0}$ follows the $(\tau,\rho)$-recursion. Thus we can write
\[
\rho_{n+1}=U_{n+1} \rho_n, ~ \tau_{n+1}=\frac{\tau_n \rho_{n+1}^2}{\tau_n Q_{n+1} +\rho_{n+1}^2}
\]
where $(U_n,Q_n)_{n \ge 0}$ is an i.i.d sequence of common law $\cU([0,1]) \otimes \chi_1^2$ independent of $(\tau_0,\rho_0)$. Now, by using the equalities in \eqref{substitution}, we see that
\[
(\tau_0,\rho_0,U_1,Q_1) = h(\tilde{\tau}_0,\tilde{\rho}_0,\tilde{U},\tilde{Q}) \ed h(\tau_0,\rho_0,U,Q)
\]
which proves the first part.
Finally, it suffices to notice that
\[
\frac{\rho_1^2}{\tau_1}=Q_1+\frac{U_1^2\rho_0^2}{\tau_0}=\frac{\tilde{\rho_0}^2}{\tilde{\tau_0}} \ed  \frac{\rho_0^2}{\tau_0}  \quad \text{ and that } \quad \frac{\rho_{n+1}^2}{\tau_{n+1}}=Q_{n+1}+\frac{U_{n+1}^2 \rho_n^2}{\tau_n}
\]
for all $n \ge 0$ 
to see that $(\frac{\rho_n}{\sqrt{\tau_n}})_{n \ge 0}$ is a stationary Markov chain. On the other hand, we also have the identity
\[
\frac{\kappa_1^2}{\tau_1}=\frac{\tilde{\kappa_0}^2}{\tilde{\tau_0}} \ed \frac{\kappa_0^2}{\tau_0}
\]
Now,  by the same reasoning as before,  conditioning on the slopes $(M_n)_{n \ge 0}$, we obtain the following identity in law
\[
(\tau_{n+1},\Delta \tau_{n+1},\kappa_{n+1},\Delta \kappa_{n+1}) \ed \left(\frac{U^2 \tau_n}{M_n^2}, \frac{U^2 Q}{M_n^2}, \frac{U \kappa_n}{M_n}, \frac{UQ}{M_n}\right)
\]
where on the right-hand side $(M_n,\tau_n,\kappa_n)$ is independent of $(U,Q)$. It follows easily that $(\frac{\kappa_n}{\sqrt{\tau_n}})_{n \ge 0}$ is stationary, while clearly being non-Markovian from the $(\tau,\rho)$-recursion equalities.
\end{proof}

We finish this section now by proving the Theorem \ref{meander} about the pseudo-meanders. We prove the statement for each of the meanders separately.
\begin{proof}[Proof of Theorem \ref{meander} for $\tilde{B}^{\text{me}}$]
From Williams path decomposition we know that 
\[
(B_{-\mu}(t) , 0 \le t \le \sigma_{\mu} \vert \sigma_{\mu}=T, B_{-\mu}(\sigma_{\mu})=y)
\]
has the same distribution as a Brownian first passage bridge from $(0,0)$ to $(T,y)$. Hence, using the representation of first passage bridges in terms of three-dimensional Bessel bridges (See \cite{bcp}), and Brownian scaling, we have that
\begin{align}
\label{bridge1}
\begin{split}
\left(\tilde{B}^{\text{me}}(u) , 0 \le u \le 1 \vert \sigma_{\mu}=T, B_{-\mu}(\sigma_{\mu})=y \right) \ed \\
\left(\text{BES}^0(3) (t) , 0 \le t \le 1 \vert \text{BES}^0(3)(1)=\frac{y}{\sqrt{T}}\right)
\end{split}
\end{align}
Hence
\begin{equation}
\label{law:me}
\left(\tilde{B}^{\text{me}}(u) , 0 \le u \le 1 \vert \tilde{B}^{\text{me}}(1)=r \right) \ed \left(\text{BES}^0(3) (t) , 0 \le t \le 1 \vert \text{BES}^0(3)(1)=r\right)
\end{equation}
It follows thus that the law of $\tilde{B}^{\text{me}}$ is absolutely continuous with respect to the law of $\text{BES}^0(3)$. From Williams path decomposition, we have that
\begin{align}\label{equ-williams}
(\mu^2 \sigma_{\mu}, \mu B(\sigma_{\mu})) \ed (\chi_3^2 \beta_{1,2}^2,\chi_3^2 \beta_{1,2})
\end{align}
where on the right hand side $\chi_3$ and $\beta_{1,2}$ are independent. Hence, it follows that $\tilde{B}^{\text{me}}(1) \ed \chi_3(1-\beta_{1,2})$, and hence its density is equal to
\[
f_{\tilde{B}^{\text{me}}(1)} (t)=4t \overline{\Phi}(t) , t >0
\]
The Radon Nikodym derivative then follows from dividing this last density by $2t^2\phi(t)$; the density of $\text{BES}^0(3)(1)\ed \chi_3$.
\end{proof}
Before proving the second part of Theorem \ref{meander} for the second meander, we will prove the following lemma
\begin{lem}
\label{drift:bridges}
Let $R$ be a three-dimensional Bessel bridge from $(0,0)$ to $(1,r)$, and let $0 \le m \le r$, then for all $k \ge 0$, we have 
\[
(R(u) + k u , 0 \le u \le 1 \vert \mathfrak{M}(R)=m) \ed (W(u), 0 \le u \le 1 \vert \mathfrak{M}(W)=m+ k)
\]
where $W$ is a three-dimensional Bessel bridge from $(0,0)$ to $(1,r+k)$. 
\end{lem}
\begin{proof}
The key idea is to provide a path decomposition of the Bessel bridge at the time it reaches its minslope. For that purpose, we use an argument of time-inversion. We introduce the process $Y$ defined as 
\[
Y(t):=(1+t)R\left (\frac{1}{1+t} \right), t \ge 0
\]
$Y$ is a $\text{BES}^r(3)$ (a three-dimensional Bessel process started at $r$). Let $S:=\text{argmin}_{t \ge 0} Y(t)$ and $\zeta=\inf_{t \ge 0} Y(t)$. By Williams path decomposition of a Bessel process when it reaches its ultimate minimum (see \cite{wil}), we have that given $\{\zeta=m\}$, the two processes 
\[
(X^1(t):=Y(S+t)-\zeta , t \ge 0) \text{ and } (X^2(t):=Y(t), 0 \le t \le S)
\]
are independent, with $X^1$ is a $\text{BES}^0(3)$, while $X^2$ is a Brownian motion started at $r$ and killed when it first hits $m$. As
\[
\mathfrak{B}(R)=\frac{1}{1+S} \text{ and } \mathfrak{M}(R)=\zeta
\]
By conditioning further on $\{ S=T\}$, the law of $X^1$ remains that of $\text{BES}^0(3)$, whereas $X^2$ has now the law of a first passage bridge between $(0,r)$ and $(T,m)$, and by using the theorem in the introduction of \cite{bcp}, we get that given $\{S=T, \zeta=m\}$, $X^2-m$ is a three-dimensional Bessel bridge from $(0,r-m)$ to $(T,0)$. Hence given $\{S=T,\zeta=m\}=\{ \mathfrak{B}(R)=\beta,\mathfrak{M}(R)=m\}$ where $\beta=\frac{1}{1+T}$, we have that
\begin{align}\label{eq-law}
\begin{split}
(R(s)-ms, 0 \le s \le \beta |  \mathfrak{B}(R)=\beta,\mathfrak{M}(R)=m) = \\
\left(sX^{1}\left(\frac{\beta-s}{\beta s}\right), 0 \le s \le \beta \vert S=T,\zeta=m \right)
\end{split}
\end{align}
However, by the path transformation that maps the Bessel processes to Bessel bridges (see the end of \cite{pit83}), the process on the right-hand side of \eqref{eq-law} is a Brownian excursion of length $\beta$. Similarly, we have that
\begin{align} \label{eq-law2}
\begin{split}
(R(s+\beta)-R(\beta)-ms,0 \le s \le 1-\beta(R) | \mathfrak{B}(R)=\beta,\mathfrak{M}(R)=m ) =\\
\left (\left(s+\beta \right)X^{2}\left(\frac{1-\beta-s}{s+\beta}\right), 0 \le s \le 1-\beta \vert S=T,\zeta=m \right)
 \end{split}
\end{align}
and the law of the right-hand side process is that of a Bessel bridge between $(0,0)$ and $(1-\beta,r-m)$.
Applying this path decomposition to both the processes $R$ and $W$ we get that 
\begin{align}
\begin{split}
(R(u)+k u ,0 \le u \le 1 \vert \mathfrak{M}(R)=m, \mathfrak{B}(R)=\beta) \ed \\
(W(u) , 0 \le u \le 1 \vert \mathfrak{M}(W)=m+k, \mathfrak{B}(W)=\beta)
\end{split}
\end{align}
To finish our proof, it suffices to prove that :
\[
(\mathfrak{B}(R) \vert \mathfrak{M}(R)=m) \ed (\mathfrak{B}(W) \vert \mathfrak{M}(W)=m+k)
\]
However by referring to the time-inversion argument using $Y$ and its path decomposition, proving this assertion is the same as showing that the distribution of the hitting time of a BM started at $r$ (resp. $r+k$) of the level $m$ (resp. $m+k$) are the same, which is true.
\end{proof}

We can now finish the proof of Theorem \ref{meander}
\begin{proof}[Proof of Theorem \ref{meander} for $\hat{B}^{\text{me}}$]
By definition we have that
\[
\check{B}^{\text{me}}(u)=\tilde{B}^{\text{me}}(u)+\mu \sqrt{\sigma_{\mu}} u ,~ 0 \le u \le 1
\]
Let us condition on the value of $\{B_{-\mu}(\sigma_{\mu}),\sigma_{\mu},\mathfrak{M}(\hat{B}^{\text{me}})\}$, we have then that 
\begin{align*}
(\check{B}^{\text{me}}(u) , 0 \le u \le 1 \vert B_{-\mu}(\sigma_{\mu})=y,\sigma_{\mu}=T,\mathfrak{M}(\hat{B}^{\text{me}})=m) =\\
(\tilde{B}^{\text{me}}(u)+\mu \sqrt{T} u , 0 \le u \le 1 | B_{-\mu}(\sigma_{\mu})=y,\sigma_{\mu}=T,\mathfrak{M}(\tilde{B}^{\text{me}})=m-\mu \sqrt{T})
\end{align*}
Using the equality in law in \eqref{law:me}, we deduce that
\begin{align*}
(\check{B}^{\text{me}}(u) , 0 \le u \le 1 \vert B_{-\mu}(\sigma_{\mu})=y,\sigma_{\mu}=T,\mathfrak{M}(\hat{B}^{\text{me}})=m) \ed \\
 (Y(u)+\mu \sqrt{T} u , 0 \le u \le 1 \vert Y(1)=\frac{y}{\sqrt{T}} , \mathfrak{M}(Y)=m- \mu\sqrt{T})
\end{align*}
where $Y \ed \text{BES}^0(3)$. Using now Lemma \ref{drift:bridges}, we have that
\begin{align*}
(\check{B}^{\text{me}}(u) , 0 \le u \le 1 \vert B_{-\mu}(\sigma_{\mu})=y,\sigma_{\mu}=T,\mathfrak{M}(\hat{B}^{\text{me}})=m) \ed \\
(Y(u) , 0 \le u \le 1 \vert Y(1)=\frac{y}{\sqrt{T}}+\mu \sqrt{T} ,\mathfrak{M}(Y)=m)
\end{align*}
Now, since $\hat{B}^{\text{me}}(1)=\frac{y}{\sqrt{T}}+\mu \sqrt{T}$ given that $\{B_{-\mu}(\sigma_{\mu})=y,\sigma_{\mu}=T\}$, we deduce that
\begin{align*}
(\check{B}^{\text{me}}(u) , 0 \le u \le 1 \vert \hat{B}^{\text{me}}(1)=r,\mathfrak{M}(\hat{B}^{\text{me}})=m) \ed \\
(Y(u) , 0 \le u \le 1 \vert Y(1)=r,\mathfrak{M}(Y)=m)
\end{align*}
Hence, $\hat{B}^{\text{me}}$ is absolutely continuous with respect to the law of the three-dimensional Bessel process on $[0,1]$. To find the expression of the Radon-Nikodym it suffices the compute the ratio of both the joint densities of $(\hat{B}^{\text{me}}(1),\mathfrak{M}(\hat{B}^{\text{me}}))$ and $(Y(1),M(Y))$. 
First, see that from the arguments of the proof of the first part of Theorem \ref{meander}, given that $Y(1)=r$, $\mathfrak{M}(Y)$ has the same distribution as the infimum of a $\text{BES}^r(3)$ which is uniform on $[0,r]$, hence 
\begin{align}\label{ind-minslope}
\left(Y(1),\frac{\mathfrak{M}(Y)}{Y(1)}\right) \ed (\chi_3,U)
\end{align}
where $\chi_3$ and $U$ are independent, with $U \sim \mathcal{U}([0,1])$. On the other hand, from \eqref{bridge1}, we have that given $(\sigma_{\mu},B_{-\mu}(\sigma_{\mu}))$, $\tilde{B}^{\text{me}}$ is a Bessel bridge between $(0,0)$ and $(1,\check{B}^{\text{me}})=(1,\frac{B_{-\mu}(\sigma_{\mu})}{\sqrt{\sigma_{\mu}}})$, hence by using the observation in \eqref{ind-minslope}, we get that $U_{\mu}:=\frac{\mathfrak{M}(\tilde{B}^{\text{me}})}{\tilde{B}^{\text{me}}(1)}$ is independent from $(\sigma_{\mu},B_{-\mu}(\sigma_{\mu}))$ and is uniform on $[0,1]$, however 
\[
\mathfrak{M}(\hat{B}^{\text{me}})=\mathfrak{M}(\tilde{B}^{\text{me}})+\mu \sqrt{\sigma_{\mu}}=U_{\mu}\frac{B_{-\mu}(\sigma_{\mu})}{\sqrt{\sigma_{\mu}}}+\mu \sqrt{\sigma_{\mu}}=U_{\mu}\frac{B(\sigma_{\mu})}{\sqrt{\sigma_{\mu}}}+\mu \sqrt{\sigma_{\mu}}(1-U_{\mu})
\]
Using the factorization of Proposition in \eqref{equ-williams}, we get that :
\[
\mathfrak{M}(\hat{B}^{\text{me}})=U_{\mu}\chi_3+(1-U_{\mu})\chi_3 \beta_{1,2} \text{  and } \hat{B}^{\text{me}}_1=\chi_3
\]
where all the random variables appearing are independent, thus
\[
\left(\hat{B}^{\text{me}}_1,\frac{\mathfrak{M}(\hat{B}^{\text{me}})}{ \hat{B}^{\text{me}}_1}\right) = (\chi_3,\beta_{1,2}(1-U_{\mu})+U_{\mu}) \ed(\chi_3, \beta_{2,1}) 
\]
Dividing the joint densities of those two pairs gives us the expression desired of the density of $\hat{B}^{\text{me}}$ with respect to $Y$.
\end{proof}
\section{$2K-B$ conjecture}\label{2k-b}
As pointed out in the introduction, the processes $2K-B$ and $\text{BES}^{0}(5)$ have the same one-dimensional marginal distribution at each time $t \ge 0$, that is the distribution of $\sqrt{t} \chi_5$. 
Not only that, these two processes share some more common properties, such as having the same quadratic variation, and being invariant under both Brownian scaling and time inversion. 
The hope behind the previous Markovian analysis was to use the process $\Psi$ and the theory of Markov functions (see \cite{pitroger81}) to show that these two processes were identical in law.
However, the complexity of the transition probabilities of the process $\Psi$ makes checking the intertwining condition a very complicated task. The following proposition further support the conjecture \ref{conjec} by showing that, in some sense, the processes $2K-B$ and $\text{BES}^{0}(5)$ have \textit{locally} the same law.
\begin{prop}\label{local-2k-b}
Let $\varphi$ be a  twice continuously differentiable function with compact support on $(0,\i)$, and $z >0$ then 
\begin{align*}
\lim_{h \downarrow 0} \frac{1}{h} \left(\bE[\varphi(2K(1+h)-B(1+h))-\varphi(2K(1)-B(1)) | 2K(1)-B(1)=z]\right)=\\ 
\frac{2}{z}\varphi'(z)+\frac{1}{2} \varphi''(z)
\end{align*}
\end{prop}
\begin{proof}
By conditioning on the value of the $4$-tuple $(K'(1),K(1),K(1)-B(1),D_1-1)=(a,k,y,\omega)$ where $D_1$ is the first vertex of $K$ after time $1$. We need to find
\begin{align*}
\lim_{h \downarrow 0} \frac{1}{h} \int (\bE[\varphi(2K(1+h)-B(1+h))-\varphi(2K(1)-B(1)) \\
\Big| K'(1)=a,K(1)=k,K(1)-B(1)=y,D_1-1=\omega)\times \\ \bP(K'(1) \in da,K(1)-B(1) \in dy,D_1-1 \in d\omega | 2K(1)-B(1)=z)])
\end{align*}
with $k=z-y$.  By Theorem \ref{4-dim},  we have an explicit description of the semi-group of the process $\Psi$. In particular, its transition probabilities have a simpler form for short times (before encountering the next vertex of the concave majorant). More precisely, we have that the law of the process
\begin{equation}\label{law}
\left(K(1+u)-B(1+u),0 \le u \le \omega | (K'(1),K(1),K(1)-B(1),D_1-1)=(a,k,y,\omega)\right)
\end{equation}
 is that of a three-dimensional Bessel bridge from $(0,y)$ to $(\omega,0)$. We denote the process in \eqref{law} by $Z$. By the space-time transform that maps Bessel bridges to Bessel processes, the process $Z$ has the same distribution as
\[
(Z(u) , 0 \le u \le \omega) \ed (\sqrt{\omega} \left(1-\frac{u}{\omega}\right)R\left(\frac{u}{\omega-u}\right) , 0 \le u \le \omega)
\]
where $R$ is a three-dimensional Bessel process started from $\frac{y}{\sqrt{\omega}}$ at time zero. As $K$ is linear on the interval $[1,1+\omega]$, it follows that conditionally on 
\[
A:=\{K'(1)=a,K(1)=k,K(1)-B(1)=y,D_1-1=\omega\}
\]
we have the following equality in law
\[
(2K(1+u)-B(1+u) , 0 \le u \le \omega) \ed (k+au +\sqrt{\omega} \left(1-\frac{u}{\omega}\right)R\left(\frac{u}{\omega-u}\right) , 0 \le u \le \omega)
\]
Hence for $h$ small enough, we have
\begin{align*}
\bE[\varphi(2K(1+h)-B(1+h)) | A]= \\
\bE\left [\varphi \left(k+ah+\sqrt{\omega} \left(1-\frac{h}{\omega}\right)R\left(\frac{h}{\omega-h}\right)\right) | R(0)=\frac{y}{\sqrt{\omega}}\right]
\end{align*}
However, almost surely
\begin{align*}
\varphi \left(k+ah+\sqrt{\omega} \left(1-\frac{h}{\omega}\right)R\left(\frac{h}{\omega-h}\right)\right) -\varphi \left(k+\sqrt{\omega}R\left(\frac{h}{\omega-h}\right)\right)=\\
\left(a-\frac{1}{\sqrt{\omega}}R\left(\frac{h}{\omega-h}\right)\right)h\varphi'(\xi_h)
\end{align*}
for a random $\xi_h$ such that
\[
\left|\xi_h- \left(k+\sqrt{\omega}R\left(\frac{h}{\omega-h}\right)\right)\right| \le \left(a-\frac{1}{\sqrt{\omega}}R \left(\frac{h}{\omega-h} \right)\right)h
\]
Hence $\xi_h \rightarrow k+y=z$ almost surely as $h \to 0$. Hence we have that
\begin{equation}\label{1}
\begin{split}
\bE[\varphi(2K(1+h)-B(1+h)) | A]=\\
\bE\left[\varphi \left(k+\sqrt{\omega}R\left(\frac{h}{\omega-h}\right)\right)|R(0)=\frac{y}{\sqrt{\omega}}\right]+\left(a-\frac{y}{\omega}\right)\varphi'(z)h+o(h)
\end{split}
\end{equation}
The error is controlled by the fact that $\varphi$ has a compact support, and the stochastic continuity of the Bessel process $R$. Now, the Bessel process $R$ is known to verify the following SDE 
\[
dR(t)=d\beta(t)+\frac{dt}{R(t)}
\]
where $\beta$ is a Brownian motion. Hence, by It\^{o} formula for any $C^2$ function $\psi$, we have
\[
d\psi(R(t))=\psi'(R(t))d\beta(t)+\left(\frac{\psi'(R(t))}{R(t)}+\frac{1}{2}\psi''(R(t))\right)dt
\]
Hence
\[
\psi \left(R\left(\frac{h}{\omega-h}\right)\right)-\psi(R(0))=\int_{0}^{\frac{h}{\omega-h}} \psi'(R(s))d\beta(s) + \int_{0}^{\frac{h}{\omega-h}} \left(\frac{\psi'(R(s))}{R(s)}+\frac{1}{2}\psi''(R(s))\right)ds
\]
By taking the expectation, we get that
\[
\bE\left[\psi \left(R\left(\frac{h}{\omega-h}\right)\right)\vert R(0)=\frac{y}{\sqrt{\omega}}\right]=\psi \left(\frac{y}{\sqrt{\omega}}\right)+\frac{h}{\omega}\left(\frac{\psi'\left(\frac{y}{\sqrt{\omega}}\right)}{\frac{y}{\sqrt{\omega}}}+\frac{1}{2}\psi''\left(\frac{y}{\sqrt{\omega}}\right)\right)+O(h^2)
\]
Now, by taking $\psi(\cdot)=\varphi(k+\sqrt{\omega} \times \cdot)$, we get
\[
\bE\left[\varphi \left(k+\sqrt{\omega}R\left(\frac{h}{\omega-h}\right)\right)|R(0)=\frac{y}{\sqrt{\omega}}\right]=\varphi(z)+\frac{h}{\omega} \left(\omega\frac{\varphi'(z)}{y}+\frac{1}{2}\omega \varphi ''(z)\right)+O(h^2)
\]
which can be simplified to
\begin{equation}\label{2}
\bE\left[\varphi \left(k+\sqrt{\omega}R\left(\frac{h}{\omega-h}\right)\right)|R(0)=\frac{y}{\sqrt{\omega}}\right]=\varphi(z)+h\left(\frac{\varphi'(z)}{y}+\frac{1}{2} \varphi ''(z)\right)+O(h^2)
\end{equation}
Hence from \eqref{1} and \eqref{2}, we get
\[
\bE[\varphi(2K(1+h)-B(1+h)) | A]=\varphi(z)+h\left(\left(a+\frac{1}{y}-\frac{y}{\omega}\right)\varphi'(z)+\frac{1}{2}\varphi''(z)\right)+O(h^2)
\]
To finish the proof, we need thus to evaluate
\[
\bE \left[K'(1)+\frac{1}{K(1)-B(1)}-\frac{K(1)-B(1)}{D_1-1}\Big|2K_1-B_1=z \right]
\]
Now, from Proposition \ref{fixed-time}, we have that the joint density of $(K'(1),K(1)-K'(1),K(1)-B(1))$ at $(a,b,y)$ is given by
\[
f_{(K'(1),K(1)-K'(1),K(1)-B(1))}(a,b,y)=4y(a+b+y)\phi(a+b+y)\ind_{a,b,y>0}
\]
It follows then that the joint density of $(K'(1),K(1)-B(1),2K(1)-B(1))$ at $(a,y,z)$ is given by
\[
f_{(K'(1),K(1)-B(1),2K(1)-B(1))}(a,b,z)=4yz\phi(z) \ind_{a+y<z, a,y>0}
\]
By integrating out separately both $y$ and $a$, we get that
\[
f_{(K'(1),2K(1)-B(1))}(a,z)=2(z-a)^2z\phi(z) \ind_{0<a<z}
\]
and
\[
f_{(K(1)-B(1),2K(1)-B(1))}(y,z)=4y(z-y)z\phi(z) \ind_{0<y<z}
\]
However, it was shown that the density of $2K(1)-B(1)$ is that of $\chi_5$, i.e $f_{2K(1)-B(1)}(z)=\frac{2}{3}z^4\phi(z)$. Hence,
\[ 
f_{K'(1) | 2K(1)-B(1)=z}(a)=3\frac{(z-a)^2}{z^3}\ind_{0<a<z}\]
and
\[
f_{K(1)-B(1)|2K(1)-B(1)=z}(y)=\frac{6y(z-y)}{z^3} \ind_{0<y<z}
\]
Thus
\[
\bE[K'(1) | 2K(1)-B(1)=z]=\frac{z}{4} \text{ and } \bE\left[\frac{1}{K(1)-B(1)} | 2K(1)-B(1)=z\right]=\frac{3}{z}
\]
However, from Proposition \ref{fixed-time}, we have that the density of $D_1-1$ at $\omega$ given that $(K'(1),K(1)-B(1),2K(1)-B(1))=(a,y,z)$ is 
\[
\frac{a(y+(z-a)\omega)}{\sqrt{2\pi \omega^3} z} \exp \left(-\frac{(y-a\omega)^2}{2\omega} \right)
\]
and as the joint density of $(K'(1),K(1)-B(1))$ at $(a,y)$ given that $2K(1)-B(1)=z$ is equal to
\[
f_{(K'(1),K(1)-B(1))|2K(1)-B(1)=z}(a,y)=\frac{6y}{z^3}\ind_{a+y<z,a,y>0}
\]
then the density of $(K'(1),K(1)-B(1),D_1-1)$ at $(a,y,\omega)$ given that $2K(1)-B(1)=z$ is given by
\[
f_{(K'(1),K(1)-B(1),D_1-1)|2K(1)-B(1)=z}(a,y,\omega)=\frac{6ay(y+(z-a)\omega)}{\sqrt{2\pi \omega^3}z^4}  \exp \left(-\frac{(y-a\omega)^2}{2\omega} \right)
\]
To finish the proof, we just need to evaluate the following integral
\[
\int \int \int \frac{6ay^2(y+(z-a)\omega)}{\sqrt{2\pi \omega^5}z^4}  \exp \left(-\frac{(y-a\omega)^2}{2\omega} \right)\ind_{y+a<z ,y,a,\omega>0} dy da d\omega
\]
By integrating out first with respect to $\omega$ we get the expression
\[
\int_{0}^{z} da \int_{0}^{z-a} \left(\frac{6a}{z^4}+\frac{6ay}{z^3}\right)dy
\]
which is easily seen to give the value
\[
\frac{1}{z}+\frac{z}{4}
\]
This ends the proof.
\end{proof}
\begin{remark}
It is well known that the generator of $\text{BES}^{0}(5)$ is given by the differential operator $\frac{2}{z} \frac{d}{dz}+\frac{1}{2} \frac{d^2}{dz^2}$. Hence, the only remaining part of the conjecture that is left to prove is to show that $2K-B$ is a Feller process (a Markov process with sufficiently regular semigroups). As was pointed out in the discussion prior to the statement of Proposition \ref{local-2k-b}, applying the theory of Markov functions on the process $\Psi$ and the projection map $\pi : (a,k,y,\omega) \mapsto k+y$ is challenging due to the complicated form of the semigroup of $\Psi$. The second part  of the semigroup (after the next vertex of the concave majorant) that involves the convex minorant of the three-dimensional Bessel process with drift, is particularly involved as is seen from the densities in Proposition \ref{drift-fixed}. In theory, one should be able to express the joint law of $(2K(s)-B(s),2K(t)-B(t))$ for any $s<t$ as a six-dimensional integral from the results of Theorem \ref{4-dim} and Proposition \ref{drift-fixed}, but this seems out of reach of our computational expertise. \\

 Another difficulty that arises in trying to prove the Markovian property is the filtering problem associated with the canonical filtration of the process $2K-B$. Indeed, given the $\sigma$-algebra generated by $\{2K(u)-B(u) : u \le t\}$, it is not clear at all how to extract any information on either the path of $B$ or $K$ on $[0,t]$, this is mainly due to the fact that $K$ depends on the whole path of $B$ and the algebraic structure of the linear combination $2K-B$. In comparison, the $2M-B$ statement can be proved by using the fact that the map
\begin{align*}
\Phi& : C_{\infty}([0,\infty)) \to C_{\infty}([0,\infty)) \\
f & \mapsto \Phi(f)(t)=2\max_{0 \le s \le t}f(s)-f(t)
\end{align*}
where $C_{\infty}([0,\infty))$ is the space of continuous functions on $[0,\infty)$ such that $\sup_{t \ge 0} f(t)=+\infty$, admits an inverse given by
\begin{align*}
\Phi^{-1}& : C_{\infty}([0,\infty)) \to C_{\infty}([0,\infty)) \\
f & \mapsto \Phi(f)(t)=2\inf_{s \ge t}f(s)-f(t)
\end{align*}
We refer the reader to \cite{revuzyor}[Theorem 3.5] for details about the proof. Unfortunately, such an inversion is not possible in the $2K-B$ setting, and so it is not clear what part of the information is lost when we go from the filtration $\{K(u),B(u) : u \le t\}$ to $\{2K(u)-B(u) : u \le t \}$. One approach to analyze $2K-B$ is to look for a filtration $(\cG_t)_{t \ge 0}$ in which this process is a semi-martingale with a manageable decomposition. One candidate would be the following
\[
\cG_t:=\cF_t \bigvee \{ K(u) : u \ge 0\} \text{ for } t \ge 0
\]
where $(\cF_t)_{t \ge 0}$ is the canonical Brownian filtration. In this filtration, $K$ is an adapted and continuous increasing process, and $K-B$ is a concatenation of Brownian excursions whose lengths are determined by $K$. Hence, both $K-B$ and $2K-B$ are semi-martingales in this filtration. However, it does not seem easy to settle wether or not $2K-B$ is Markov from this perspective, and one would perhaps need to consider a more subtly enlarged filtration instead of $(\cG_t)_{t \ge 0}$ above. \\
 
As a final remark, by well known properties of the usual time inversion map from Brownian and Bessel processes to
corresponding bridges, discussed in \cite{MR2798000}, 
our conjecture that $2K-B$ has the same distribution as a five dimensional Bessel process is equivalent to the following conjecture regarding the least concave  majorant $k$ of $b$,  a standard Brownian bridge
of length $1$ from $0$ to $0$: that $2k - b$ has the the same distribution as a five dimensional Bessel bridge of length $1$ from $0$ to $0$. 
\end{remark}
\bibliographystyle{abbrv}
\bibliography{bib-cm}
\end{document}